\documentclass[12pt,reqno,b5paper]{amsart}%
\usepackage{anysize}
\usepackage[noadjust]{cite}
\usepackage[T1]{fontenc}
\usepackage{txfonts,amsmath,esint}
\usepackage[colorlinks,urlcolor=blue,citecolor=blue,linkcolor=blue]{hyperref}
\usepackage{amsmath}
\usepackage{amsfonts}
\usepackage{amssymb}
\usepackage{graphicx}%
\providecommand{\U}[1]{\protect\rule{.1in}{.1in}}
\DeclareSymbolFont{CM}{OMX}{cmex}{m}{n}
\DeclareMathSymbol{\sumop}{\mathop}{CM}{"50}
\renewcommand{\sum}{\sumop}
\marginsize{15mm}{15mm}{5mm}{5mm}

\newtheorem{theorem}{\sc Theorem}[section]

\newtheorem{lemma}[theorem]{\sc Lemma}

\newtheorem{proposition}[theorem]{\sc Proposition}
\theoremstyle{definition}

\theoremstyle{remark}
\newtheorem{remark}[theorem]{Remark}

\numberwithin{equation}{section}
\allowdisplaybreaks
\AtEndDocument{\bibliographystyle{numbst}
\bibliography{C:/Users/lao/Desktop/all}\ }
\begin{document}

\title{On quasilinear elliptic problems with finite or infinite potential wells}
\author{Shibo Liu}
\dedicatory{Department of Mathematics, Xiamen University\\Xiamen 361005, China}
\thanks{This work was supported by NSFC (11671331, 11971436) and a special fund from Ministry of Education of China (20180707). It was completed while the author was visiting the Abdus Salam
International Centre for Theoretical Physics (ICTP), he would like to thank ICTP for the hospitality.}
\maketitle

\begin{abstract}
We consider quasilinear elliptic problems of the form
\[
-\operatorname{div}\big(\phi(|\nabla u|)\nabla u\big)+V(x)\phi
(|u|)u=f(u)\qquad u\in W^{1,\Phi}(\mathbb{R}^{N}),
\]
where $\phi$ and $f$ satisfy suitable conditions. The positive potential $V\in
C(\mathbb{R}^{N})$ exhibits a finite or infinite potential well in the sense
that $V(x)$ tends to its supremum $V_{\infty}\le+\infty$ as $|x|\to\infty$.
Nontrivial solutions are obtained by variational methods. When $V_{\infty
}=+\infty$, a compact embedding from a suitable subspace of $W^{1,\Phi
}(\mathbb{R}^{N})$ into $L^{\Phi}(\mathbb{R}^{N})$ is established, which
enables us to get infinitely many solutions for the case that $f$ is odd.
For the case that $V(x)=\lambda a(x) + 1$ exhibits a steep potential well
controlled by a positive parameter $\lambda$, we get nontrivial solutions for
large $\lambda$.

\end{abstract}

\section{Introduction}

In this paper we consider the following quasilinear elliptic problem in
$\mathbb{R}^{N}$,%
\begin{equation}
-\operatorname{div}(\phi(|\nabla u|)\nabla u)+V(x)\phi(\left\vert u\right\vert
)u=f(u)\text{,\qquad}u\in W^{1,\Phi}(\mathbb{R}^{N})\text{.}\label{e1}%
\end{equation}
where $\phi:[0,\infty)\rightarrow\lbrack0,\infty)$ is a $C^{1}$-function
satisfying the following assumptions:

\begin{itemize}
\item[$(\phi_{1})$] the function $t\mapsto\phi(t)t$ is increasing in
$(0,\infty)$,

\item[$(\phi_{2})$] there exist $\ell,m\in(1,N)$ such that%
\begin{equation}
\ell\leq\frac{\phi(\left\vert t\right\vert )t^{2}}{\Phi(t)}\leq m\text{\qquad
for all }t\neq0\text{,}\label{e0}%
\end{equation}
where $\ell\leq m<\ell^{\ast}$ (note that for $p\in(1,N)$ we set $p^{\ast}=Np/(N-p)$),
\[
\Phi(t)=\int_{0}^{\left\vert t\right\vert }\phi(s)s\,\mathrm{d}s\text{.}%
\]

\end{itemize}

Nonlinear elliptic problems in $\mathbb{R}^{N}$ like (\ref{e1}) have been
extensively studied. For example, if $\phi(t)\equiv1$, then the problem
(\ref{e1}) reduces to the following stationary Schr\"{o}dinger equation%
\begin{equation}
-\Delta u+V(x)u=f(u)\text{,\qquad}u\in H^{1}(\mathbb{R}^{N})\text{,}%
\label{esh}%
\end{equation}
which is a central topic in nonlinear analysis in the last decads, see
\cite{MR1181725,MR2557725,MR1751952,MR2271695,MR2957647,MR1162728} and the
reference therein. If $\phi(t)=t^{p-2}$, then the leading term in (\ref{e1})
is the $p$-Laplacian operator $-\Delta_{p}$ and the corresponding problem has
also been studied in many papers such as
\cite{MR2794422,MR3008324,MR3454625,MR3619241}. If $\phi(t)=t^{p-2}+t^{q-2}$,
the leading term in (\ref{e1}) is the so-called $\left(  p,q\right)
$-Laplacian operator and results for the corresponding problems can be found
in \cite{MR2524439,MR3300789,MR3462564}.

For general $\phi$ satisfying $(\phi_{1})$ and $(\phi_{2})$, $-\Delta_{\Phi
}u:=-\operatorname{div}(\phi(\left\vert \nabla u\right\vert )\nabla u)$ is
called the $\Phi$-Laplacian of $u$. The $\Phi$-Laplacian operator
$-\Delta_{\Phi}$ arises in some applications such as nonlinear elasticity,
plasticity and non-Newtonian fluids. Elliptic boundary value problems
involving the $\Phi$-Laplacian have been studied on a bounded domain
$\Omega\subset\mathbb{R}^{N}$ in several recent papers, such as Cl\'{e}ment \emph{et
al} \cite{MR1777463}, Fukagai-Narukawa \cite{MR2317653} and Carvalho \emph{et al}
\cite{MR3306384}.

For unbounded domains such as $\mathbb{R}^{N}$, there are also some recent
results on the quasilinear $\Phi$-Laplacian problem (\ref{e1}). In Alves \emph{et al}
\cite{MR3328350}, the authors studied the problem (\ref{e1}) by variational
methods under the following conditions on the potential $V$ and the
nonlinearity $f$.

\begin{itemize}
\item[$(V_{0})$] $V\in C(\mathbb{R}^{N})$, $V_{0}=\inf_{\mathbb{R}^{N}}V>0$.

\item[$(f_{1})$] $f\in C(\mathbb{R})$ satisfies
\begin{equation}
\lim_{\left\vert t\right\vert \rightarrow0}\frac{f(t)}{\phi(\left\vert
t\right\vert )t}=\lambda_{0}\text{,\qquad}\lim_{\left\vert t\right\vert
\rightarrow\infty}\frac{f(t)}{\phi_{\ast}(\left\vert t\right\vert
)t}=0\text{,}%
\label{elm}
\end{equation}
where $\phi_{\ast}$ is related to $\Phi_{\ast}$, the Sobolev conjugate
function of $\Phi$ (see (\ref{esc})), via
\[
\Phi_{\ast}(t)=\int_{0}^{\left\vert t\right\vert }\phi_{\ast}(s)s\,\mathrm{d}%
s\text{.}%
\]

\item[$(f_{2})$] there exists $\theta>m$ such that for all $t\neq0$,%
\[
0<F(t):=\int_{0}^{t}f(s)\,ds\leq\frac{1}{\theta}tf(t)\text{.}%
\]
\end{itemize}
Because the problem (\ref{e1}) is settled on the unbounded domain
$\mathbb{R}^{N}$, to overcome the lack of compactness of the relevant Sobolev
embeddings, the authors considered the cases that $V$ is radial, or
$\mathbb{Z}^{N}$-periodic. Using a Strauss type result and a Lions type
concentration lemma in Orlicz-Sobolev spaces established in the paper, they
obtained nontrivial solutions for the problem via the mountain pass theorem
\cite{MR0370183}.

For the autonomous case that $V(x)\equiv0$ and $f(u)=\left\vert u\right\vert
^{s-2}-\left\vert u\right\vert ^{\alpha-2}$, nontrivial solutions for
(\ref{e1}) have also been obtained in \cite{MR3334963,MR3148112} via mountain
pass theorem, thanks to the compact embeddings from the radial Orlicz-Sobolev
spaces to certain Lebesgue spaces $L^{\tau}(\mathbb{R}^{N})$ established in
these papers. The main difference of these two papers is on the assumptions on
$\phi$.

In \cite{MR3513973}, Chorfi and R\v{a}dulescu studied the following problem%
\begin{equation}
-\operatorname{div}(\phi(\left\vert \nabla u\right\vert )\nabla
u)+a(x)\left\vert u\right\vert ^{\alpha-2}u=f(x,u)\text{\qquad in }%
\mathbb{R}^{N}\text{,}\label{ecr}%
\end{equation}
where the function $\phi$ is the same as in \cite{MR3148112} and $a$ verifies%
\begin{equation}
\lim_{\left\vert x\right\vert \rightarrow0}a(x)=+\infty\text{,\qquad}%
\lim_{\left\vert x\right\vert \rightarrow+\infty}a(x)=+\infty\text{.}%
\label{ea}%
\end{equation}
Note that the zero order term on the left hand side of (\ref{ecr}) is a power function of $u$, which is different to that of (\ref{e1}). Using the strategy initiated by Rabinowitz \cite{MR1162728}, condition
(\ref{ea}) enables the authors to overcome the lack of compactness and obtain a
nontrivial solution for the problem (\ref{ecr}).

There are also some papers for the case that there is a parameter
$\varepsilon>0$ in (\ref{e1}), existence and multiplicity of solutions for the
equation were obtained for $\varepsilon$ small, see
\cite{MR3905639,MR3688036,MR3694761}.

Our results are closely related to those of Alves \emph{et al} \cite{MR3328350}. As
mentioned before, in their paper they studied the case that the potential $V$
is radial or $\mathbb{Z}^{N}$-periodic. In our first result we consider the
case that $V$ satisfies the following condition due to Bartsch and Wang
\cite{MR1349229} in their study of (\ref{esh}).

\begin{itemize}
\item[$(V_{1})$] for all $M>0$, $\mu(V^{-1}(-\infty,M])<\infty$.
\end{itemize}
Here $\mu$ is
the Lebesgue measure on $\mathbb{R}^{N}$.
Note that $\left(  V_{1}\right)  $ is satisfied if $V$ is coercive:%
\begin{equation}
\lim_{\left\vert x\right\vert \rightarrow\infty}V(x)=+\infty\text{.}\label{ev}%
\end{equation}

To apply variational methods let $X$ be a suitable subspace of the
Orlicz-Sobolev space $W^{1,\Phi}(\mathbb{R}^{N})$ that will be made clear in
Section 2, and consider the $C^{1}$-functional $\mathcal{J}:X\rightarrow
\mathbb{R}$ given by
\[
\mathcal{J}(u)=\int_{\mathbb{R}^{N}}\Phi(|\nabla u|)+\int_{\mathbb{R}^{N}%
}V(x)\Phi(\left\vert u\right\vert )-\int_{\mathbb{R}^{N}}F(u)\text{.}%
\]
Then, solutions of (\ref{e1}) are critical points of $\mathcal{J}$.

\begin{theorem}
\label{t1}Suppose $\phi$ satisfies $(\phi_{1})$, $(\phi_{2})$; $V$ satisfies
$(V_{0})$, $(V_{1})$; $f$ satisfies $(f_{1})$, $(f_{2})$.

\begin{itemize}
\item[$\left(  1\right)  $] If $\lambda_{0}=0$, then \eqref{e1} has a
nontrivial solution.

\item[$\left(  2\right)  $] If $\lambda_{0}\geq0$ and $f$ is odd, then
\eqref{e1} has a sequence of solutions $\left\{  u_{n}\right\}  $ such that
$\mathcal{J}(u_{n})\rightarrow+\infty$.
\end{itemize}
\end{theorem}

Motivated by Bartsch and Wang \cite{MR1349229}, the assumption $\left(
V_{1}\right)  $ enables us to establish a compact embedding result from our
working space $X$ into  subcritical Orlicz spaces, see Lemma \ref{l0}. With
this result we can regain compactness for our functional $\mathcal{J}$ and
get critical points.

As a special case of $\left(  V_{1}\right)  $, (\ref{ev}) can be interpreted
as $V$ has an infinite potential well. In our next result we inverstigate the
case that $V$ exhibits a finite potential well:

\begin{itemize}
\item[$(V_{2})$] for all $x\in\mathbb{R}^{N}$, $
V(x)<V_{\infty}:=\lim\limits_{\left\vert x\right\vert \rightarrow\infty}%
V(x)<\infty$.
\end{itemize}
Under the assumption $\left(  V_{2}\right)  $ the above compact embedding is
not valid anymore. Hence, to get critical points of $\mathcal{J}$, we need the
following monotonicity assumptions on $\phi$ and $f$:

\begin{itemize}
\item[$(\phi_{3}^{s})$] for some $s\geq2$, the function $t\mapsto
\phi(t)/t^{s-2}$ is nonincreasing on $(0,\infty)$,
\item[$(f_{3}^{s})$] for some $s\geq2$, the function $t\mapsto f(t)/\left\vert
t\right\vert ^{s-1}$ is strictly increasing on $(0,\infty)$ and $(-\infty,0)$.
\end{itemize}
Note that $(f_{3}^{s})$ implies that for all $\xi\in\mathbb{R}\backslash
\{0\}$, $t\mapsto f(t\xi)/t^{s-1}$ is strictly increasing on $(0,\infty)$. Our
result reads as follows.

\begin{theorem}
\label{t2}Suppose $\phi$ satisfies $(\phi_{1})$, $(\phi_{2})$ and $(\phi
_{3}^{s})$; $V$ satisfies $(V_{0})$, $(V_{2})$; $f$ satisfies $(f_{1})$ with
$\lambda_{0}=0$, $(f_{2})$ and $(f_{3}^{s})$. Then \eqref{e1} has a nontrivial solution.
\end{theorem}

For our last result, we 
consider the case that the potential $V(x)$ is of the form $\lambda a(x)+1$ with $\lambda>0$ and $a$ satisfies

\begin{itemize}
\item[$( a_{1}) $] $a\in C( \mathbb{R}^{N}) $, $a\geq0$ and $a^{-1}( 0) $ has
nonempty interior.

\item[$(a_{2})$] for some $M_{0}>0$ we have $\mu(a^{-1}(-\infty,M_{0}%
])<\infty$.
\end{itemize}
For problem (\ref{esh}), these conditions are introduced by Bartsch and Wang \cite{MR1349229}, and characterizes $V$ as possessing a steep potential well whose
height is controlled by the positive parameter $\lambda$. Our result for this
case is the following theorem.

\begin{theorem}
\label{t3}Suppose $(\phi_{1})$, $(\phi_{2})$, $(a_{1})$ and $(a_{2})$ are satisfied, $f\in C(\mathbb{R})$ satisfies $(f_{1})$ with
$\lambda_{0}=0$ and $(f_{2})$.
Then there exists $\lambda^{\ast}>0$ such that for all $\lambda\geq
\lambda^{\ast}$, the problem%
\begin{equation}
-\operatorname{div}(\phi(|\nabla u|)\nabla u)+V_{\lambda}(x)\phi(\left\vert
u\right\vert )u=f(u)\text{,\qquad}u\in W^{1,\Phi}(\mathbb{R}^{N}%
)\text{.}\label{ee}%
\end{equation}
has a nontrivial solution, here $V_{\lambda}(x)=\lambda a(x)+1$.
\end{theorem}

Our Theorems \ref{t1} and \ref{t3} are generalizations of Theorem 2.1 and part
of Theorem 2.4 in Bartsch and Wang \cite{MR1349229}, respectively. However, even for the semilinear case that $\phi(t)\equiv1$, our Theorem \ref{t3} is slightly general than the corresponding result in \cite{MR1349229}, because in $(f_1)$ we only require $f$ to be asymptotically subcritical, that is the second limit in (\ref{elm}) holds, while  in \cite[Theorem 2.4]{MR1349229} the nonlinearity $f$ is strictly subcritical, meaning that the growth of $f$ at infinity is controlled by a subcritical power function $|t|^{q-2}t$ for some $q\in(2,2^*)$.  See remark \ref{r26} for more details.
Roughly
speaking, Theorem \ref{t2} also generalizes Rabinowitz \cite[Theorem
4.27]{MR1162728}. 

Both \cite{MR1349229} and \cite{MR1162728} are concerned on
the semilinear equation (\ref{esh}). Our $\Phi$-Laplacian equation (\ref{e1}) is much more general.

The paper is organized as follows. In Section 2 we recall some concepts and results about Orlicz spaces and prove the compact embedding lemma (Lemma \ref{l0}) mentioned before. The existence of nontrivial solutions is proved in Section 3. Finally, in Section 4 we deal with the multiplicity result stated in Theorem \ref{t1} (2).

\section{Orlicz-Sobolev spaces}

In this section, we recall some results about Orlicz spaces and Orlicz-Sobolev
spaces that we will use for proving our main results. The reader is refereed
to \cite{MR3328350,MR2271234} and the references therein, in particular
\cite{MR2424078}, for more details.

A convex, even continuous function $\Phi:\mathbb{R}\rightarrow\lbrack
0,\infty)$ is called a nice Young function, $\mathcal{N}$-function for short,
if $\Phi(t)=0$ is equivalent to $t=0$, and%
\[
\lim_{t\rightarrow0}\frac{\Phi(t)}{t}=0\text{,\qquad}\lim_{t\rightarrow
+\infty}\frac{\Phi(t)}{t}=+\infty\text{.}%
\]
The $\mathcal{N}$-function $\Phi$ satisfies the $\Delta_{2}$-condition if
there is a constant $K>0$ such that%
\[
\Phi(2t)\leq K\Phi(t)\text{\qquad for all }t\geq0\text{.}%
\]
Then, for an open subset $\Omega$ of $\mathbb{R}^{N}$, under the
natural addition and scale multiplication,%
\[
L^{\Phi}(\Omega)=\left\{  u:\Omega\rightarrow\mathbb{R}\left\vert \,u\text{ is
measurable, }\int_{\Omega}\Phi(\left\vert u\right\vert )<\infty\right.
\right\}
\]
is a vector space. Equipped with the Luxemburg norm%
\[
\left\vert u\right\vert _{\Phi}=\inf\left\{  \alpha>1\left\vert \,\int%
_{\Omega}\Phi\left(  \frac{\left\vert u\right\vert }{\alpha}\right)
\leq1\right.  \right\}  \text{,}%
\]
$L^{\Phi}(\Omega)$ is a Banach space, called Orlicz space. The Orlicz-Sobolev
space $W^{1,\Phi}(\mathbb{R}^{N})$ is the completion of $C_{0}^{\infty
}(\mathbb{R}^{N})$ under the norm%
\begin{equation}
\left\Vert u\right\Vert _{1}=\left\vert \nabla u\right\vert _{\Phi}+\left\vert
u\right\vert _{\Phi}\text{.} \label{en}%
\end{equation}
The complement function of $\Phi$, denoted by $\tilde{\Phi}$, is given by the
Legendre transformation%
\[
\tilde{\Phi}(s)=\max_{t\geq0}\left\{  st-\Phi(t)\right\}  \text{\qquad for
}s\geq0\text{.}%
\]
Then,%
\begin{equation}
\tilde{\Phi}(\Phi^{\prime}(t))\leq\Phi(2t)\text{\qquad for }t\geq0\text{,}
\label{ep1}%
\end{equation}
and we have the H\"{o}lder inequality
\begin{equation}
\int_{\Omega}|uv|\le2|u|_{\Phi}|v|_{\tilde{\Phi}} \label{eh}%
\end{equation}
for $u\in L^{\Phi}(\Omega)$ and $v\in L^{\tilde{\Phi}}(\Omega)$.

When%
\[
\int_{1}^{+\infty}\frac{\Phi^{-1}(s)}{s^{(N+1)/N}}\,\mathrm{d}s=+\infty\text{,}%
\]
the function $\Phi_{\ast}$ given by%
\begin{equation}
\Phi_{\ast}^{-1}(t)=\int_{0}^{t}\frac{\Phi^{-1}(s)}{s^{(N+1)/N}}\,\mathrm{d}s
\label{esc}%
\end{equation}
is called the Sobolev conjugate function of $\Phi$. It is known that similar
to (\ref{e0}) we have
\begin{equation}
\ell^{\ast}\leq\frac{\phi_{\ast}(\left\vert t\right\vert )t^{2}}{\Phi_{\ast
}(t)}\leq m^{\ast}\text{\qquad for all }t\neq0\text{.} \label{ep2}%
\end{equation}
It is also known that, if $\Phi$ and $\tilde{\Phi}$ satisfy the $\Delta_{2}%
$-condition, then $L^{\Phi}(\Omega)$ and $W^{1,\Phi}(\mathbb{R}^{N})$ are
reflexive and separable. Moreover,
\begin{align}
u_{n}\rightarrow u\text{ in }L^{\Phi}(\Omega)\quad &  \Longleftrightarrow
\quad\int_{\Omega}\Phi(\left\vert u_{n}-u\right\vert )\rightarrow
0\text{,}\label{e4}\\
u_{n}\rightarrow u\text{ in }W^{1,\Phi}(\mathbb{R}^{N})\quad &
\Longleftrightarrow\quad\int_{\mathbb{R}^{N}}(\Phi(\left\vert \nabla
u_{n}-\nabla u\right\vert )+\Phi(\left\vert u_{n}-u\right\vert ))\rightarrow
0\text{.} \label{e7}%
\end{align}
In addition, $\left\{  u_{n}\right\}  $ is bounded in $L^{\Phi}(\mathbb{R}%
^{N})$ if and only if $\left\{  \Phi(\left\vert u_{n}\right\vert )\right\}  $
is bounded in $L^{1}(\mathbb{R}^{N})$. This can be seem by setting $V=1$ in (\ref{e5}) below.

Let $\Psi$ be an $\mathcal{N}$-function verifying $\Delta_{2}$-condition. It is
well known that if%
\[
\varlimsup_{t\rightarrow0}\frac{\Psi( t) }{\Phi( t) }<+\infty\text{,\qquad
}\varlimsup_{\left\vert t\right\vert \rightarrow+\infty}\frac{\Psi( t) }{\Phi_*(
t) }<+\infty\text{,}%
\]
then we have a continuous embedding $W^{1,\Phi}(\mathbb{R}^{N})\hookrightarrow
L^{\Psi}(\mathbb{R}^{N})$. Moreover, if%
\begin{equation}
\lim_{\left\vert t\right\vert \rightarrow0}\frac{\Psi(t)}{\Phi(t)}%
<+\infty\text{,\qquad}\lim_{\left\vert t\right\vert \rightarrow\infty}%
\frac{\Psi(t)}{\Phi_{\ast}(t)}=0\text{,} \label{e2}%
\end{equation}
then the embedding $W^{1,\Phi}(\mathbb{R}^{N})\hookrightarrow L_{\mathrm{loc}%
}^{\Psi}(\mathbb{R}^{N})$ is compact, such $\Psi$ is call subcritical.

For the study of problem (\ref{e1}), we introduce the following subspace of
$W^{1,\Phi}(\mathbb{R}^{N})$. Assuming $( V_{0}) \,,( \phi_{1}) $ and $(
\phi_{2}) $, on the linear subspace%
\[
X=\left\{  u\in W^{1,\Phi}(\mathbb{R}^{N})\left\vert \,\int_{\mathbb{R}^{N}}V(
x) \Phi( \left\vert u\right\vert ) <\infty\right.  \right\}
\]
we equip the norm $\left\Vert u\right\Vert =\left\vert \nabla u\right\vert
_{\Phi}+\left\vert u\right\vert _{\Phi,V}$, where
\[
\left\vert u\right\vert _{\Phi,V}=\inf\left\{  \alpha>0\left\vert
\,\int_{\Omega}V( x) \Phi\left(  \frac{\left\vert u\right\vert }{\alpha
}\right)  \leq1\right.  \right\}  \text{.}%
\]
Then $( X,\left\Vert \cdot\right\Vert ) $ is a separable reflexive Banach
space, which will be simply denoted by $X$. If $V$ is bounded, then $X$ is
precisely the original Orlicz-Sobolev space $W^{1,\Phi}( \mathbb{R}^{N}) $,
the norm $\left\Vert \cdot\right\Vert $ is equivalent to the one given in
(\ref{en}).

\begin{lemma}
\label{l5}Assume that $( V_{0}) $, $( \phi_{1}) $ and $( \phi_{2}) $ hold and
for $t\geq0$ let%
\begin{equation}
\xi_{0}( t) =\min\left\{  t^{\ell},t^{m}\right\}  \text{,\qquad}\xi_{1}( t)
=\max\left\{  t^{\ell},t^{m}\right\}  \text{.} \label{e14}%
\end{equation}
Then for all $u\in X$ we have%
\begin{equation}
\xi_{0}( \left\vert u\right\vert _{\Phi,V}) \leq\int_{\mathbb{R}^{N}}V( x)
\Phi( \left\vert u\right\vert ) \leq\xi_{1}( \left\vert u\right\vert _{\Phi
,V}) \text{.} \label{e5}%
\end{equation}

\end{lemma}

\begin{proof}
According to \cite[Lemma 2.1]{MR2271234}, we have%
\begin{equation}
\xi_{0}( \rho) \Phi( t) \leq\Phi( \rho t) \leq\xi_{1}( \rho) \Phi( t)
\text{\qquad for }\rho,t\geq0\text{.} \label{e6}%
\end{equation}
Taking $\rho=\left\vert u\right\vert _{\Phi,V}$ and $t=|u( x) |/\left\vert
u\right\vert _{\Phi,V}$, we get%
\begin{align*}
\int_{\mathbb{R}^{N}}V( x) \Phi( \left\vert u\right\vert )  &  =\int%
_{\mathbb{R}^{N}}V( x) \Phi\left(  \left\vert u\right\vert _{\Phi,V}%
\frac{\left\vert u\right\vert }{\left\vert u\right\vert _{\Phi,V}}\right) \\
&  \leq\xi_{1}( \left\vert u\right\vert _{\Phi,V}) \int_{\mathbb{R}^{N}%
}V(x)\Phi\left(  \frac{\left\vert u\right\vert }{\left\vert u\right\vert
_{\Phi,V}}\right)  \leq\xi_{1}( \left\vert u\right\vert _{\Phi,V})
\end{align*}
because by the definition of $|\cdot|_{\Phi,V}$, the integral in the last line is not greater than $1$.
The first inequality in (\ref{e5}) can be proved similarly.
\end{proof}

\begin{remark}
Similar to (\ref{e6}), because of (\ref{ep2}), for%
\[
\xi_{0}^{\ast}(t)=\min\left\{  t^{\ell^{\ast}},t^{m^{\ast}}\right\}
\text{,\qquad}\xi_{1}^{\ast}(t)=\max\left\{  t^{\ell^{\ast}},t^{m^{\ast}%
}\right\}  \text{,}%
\]
we have%
\[
\xi_{0}^{\ast}(\rho)\Phi_{\ast}(t)\leq\Phi_*(\rho t)\leq\xi_{1}^{\ast}(\rho
)\Phi_{\ast}(t)\text{\qquad for }\rho,t\geq0\text{.}%
\]
Because $m<\ell^{\ast}$, using this and (\ref{e6}) we have%
\[
0<\frac{\Phi(t)}{\Phi_{\ast}(t)}\leq\frac{\Phi(1)\xi_{1}(t)}{\Phi_{\ast}%
(1)\xi^{\ast}_{0}(t)}\leq\frac{\Phi(1)}{\Phi_{\ast}(1)}\frac{t^{m}}%
{t^{\ell^{\ast}}}\rightarrow0\text{\qquad as }\left\vert t\right\vert
\rightarrow\infty\text{.}%
\]
Therefore, $\Phi$ is an $\mathcal{N}$-function verifying $\Delta_{2}$-condition
and (\ref{e2}).
\end{remark}

\begin{lemma}
\label{l0}Suppose $\phi$ satisfies $(\phi_{1})$, $(\phi_{2})$; $V$ satisfies
$(V_{0})$, $(V_{1})$. Then for any $\mathcal{N}$-function $\Psi$ verifying
$\Delta_{2}$-condition and \eqref{e2}, the embedding $X\hookrightarrow
L^{\Psi}(\mathbb{R}^{N})$ is compact. In particular, $X\hookrightarrow
L^{\Phi}(\mathbb{R}^{N})$ is compact.
\end{lemma}

\begin{proof}
Assume that $\left\{  u_{n}\right\}  $ is a sequence in $X$ such that
$u_{n}\rightharpoonup0$ in $X$, we want to show that $u_{n}\rightarrow0$ in
$L^{\Psi}(\mathbb{R}^{N})$. Firstly, we have $\left\Vert u_{n}\right\Vert \leq
C_{1}$ for some $C_{1}>0$.

For any $\varepsilon>0$, by (\ref{e2}) there is $k>0$ such that%
\[
\Psi(t)=\Psi(\left\vert t\right\vert )\leq\varepsilon\Phi_{\ast}(\left\vert
t\right\vert )\text{,\qquad for }\left\vert t\right\vert >k\text{.}%
\]
Since the embedding $X\hookrightarrow L^{\Phi_{\ast}}(\mathbb{R}^{N})$ is
continuous and $\left\Vert u_{n}\right\Vert \leq C_{1}$, we deduce that
$\left\{  u_{n}\right\}  $ in bounded in $L^{\Phi_{\ast}}(\mathbb{R}^{N})$.
Hence
\begin{equation}
\int_{\left\vert u_{n}\right\vert >k}\Psi(\left\vert u_{n}\right\vert
)\leq\varepsilon\int_{\mathbb{R}^{N}}\Phi_{\ast}(\left\vert u_{n}\right\vert
)\leq\varepsilon\xi_{1}^{\ast}(\left\vert u_{n}\right\vert _{\Phi_{\ast}})\leq
C_{2}\varepsilon\text{,} \label{ex}%
\end{equation}
for some $C_2>0$, where we have used an inequality for $\Phi_{\ast}$ similar to (\ref{e5}).

Given $M>0$ and $R>0$, set%
\begin{align*}
A_{R}  &  =\left\{  \left.  x\in\mathbb{R}^{N}\right\vert \,\left\vert
x\right\vert \geq R,V(x)\geq M\right\}  \text{,}\\
B_{R}  &  =\left\{  \left.  x\in\mathbb{R}^{N}\right\vert \,\left\vert
x\right\vert \geq R,V(x)<M\right\}  \text{.}%
\end{align*}
By the first limit in (\ref{e2}), there is $\kappa>0$ such that for
$t\in\left[  0,k\right]  $ we have $\Psi(t)\leq\kappa\Phi(t)$. Because
$\big\{  \vert u_{n}\vert _{\Phi,V}\big\}  $ is bounded, we can
take $M>0$ large enough such that%
\[
\frac{\kappa}{M}\xi_{1}(\left\vert u_{n}\right\vert _{\Phi,V})<\varepsilon
\text{.}%
\]
It follows from (\ref{e2}) that $\Psi$ is bounded in $\left[  0,k\right]  $.
By $(V_{1})$, we have $\mu(B_{R})\rightarrow0$ as $R\rightarrow\infty$, thus
we can choose $R>0$ such that%
\[
\mu(B_{R})\cdot\sup_{\left[  0,k\right]  }\Psi<\varepsilon\text{.}%
\]
Using the above inequalities and Lemma \ref{l5} we have%
\begin{align*}
\int_{A_{R}\cap\left\{  \left\vert u_{n}\right\vert \leq k\right\}  }%
\Psi(\left\vert u_{n}\right\vert )  &  \leq\kappa\int_{A_{R}\cap\left\{
\left\vert u_{n}\right\vert \leq k\right\}  }\Phi(\left\vert u_{n}\right\vert
)\\
&  \leq\frac{\kappa}{M}\int_{\mathbb{R}^{N}}V(x)\Phi(\left\vert u_{n}%
\right\vert )\leq\frac{\kappa}{M}\xi_{1}(\left\vert u_{n}\right\vert _{\Phi
,V})\leq\varepsilon\text{,}\\
\int_{B_{R}\cap\left\{  \left\vert u_{n}\right\vert \leq k\right\}  }%
\Psi(\left\vert u_{n}\right\vert )&\leq\mu(B_{R})\cdot\sup_{\left[  0,k\right]
}\Psi<\varepsilon\text{.}%
\end{align*}%
Consequently,%
\begin{align*}
\int_{\left\vert u_{n}\right\vert \leq k}\Psi(\left\vert u_{n}\right\vert )
&  \leq\left(  \int_{\left\vert x\right\vert \leq R}+\int_{A_{R}\cap\left\{
\left\vert u_{n}\right\vert \leq k\right\}  }+\int_{B_{R}\cap\left\{
\left\vert u_{n}\right\vert \leq k\right\}  }\right)  \Psi(\left\vert
u_{n}\right\vert )\\
&  \leq\int_{\left\vert x\right\vert \leq R}\Psi(\left\vert u_{n}\right\vert
)+2\varepsilon\text{.}%
\end{align*}
Now using (\ref{ex}) we get%
\begin{align}
\int_{\mathbb{R}^{N}}\Psi(\left\vert u_{n}\right\vert )  &  =\int_{\left\vert
u_{n}\right\vert >k}\Psi(\left\vert u_{n}\right\vert )+\int_{\left\vert
u_{n}\right\vert \leq k}\Psi(\left\vert u_{n}\right\vert )\nonumber\\
&  \leq(C_{2}+2)\varepsilon+\int_{\left\vert x\right\vert \leq R}%
\Psi(\left\vert u_{n}\right\vert )\text{.} \label{e3}%
\end{align}
Since the embedding $X\hookrightarrow L_{\mathrm{loc}}^{\Psi}(\mathbb{R}^{N})$
is compact, from $u_{n}\rightharpoonup0$ in $X$ and (\ref{e3}) we have%
\[
\varlimsup_{n\rightarrow\infty}\int_{\mathbb{R}^{N}}\Psi(\left\vert
u_{n}\right\vert )\leq(C_{2}+2)\varepsilon\text{.}%
\]
Because $\varepsilon$ is arbitrary, this implies%
\[
\int_{\mathbb{R}^{N}}\Psi(\left\vert u_{n}\right\vert )\rightarrow0
\]
and $u_{n}\rightarrow0$ in $L^{\Psi}(\mathbb{R}^{N})$.
\end{proof}

\section{Nontrivial solutions}

From now on, we assume the conditions $(\phi_{1})$, $(\phi_{2})$, $(V_{0})$
and $(f_{1})$. Then, the functional $\mathcal{J}:X\rightarrow\mathbb{R}$ given
by%
\begin{equation}
\mathcal{J}(u)=\int_{\mathbb{R}^{N}}\Phi(|\nabla u|)+\int_{\mathbb{R}^{N}%
}V(x)\Phi(\left\vert u\right\vert )-\int_{\mathbb{R}^{N}}F(u) \label{ecc}%
\end{equation}
is of class $C^{1}$. The derivative of $\mathcal{J}$ is given by%
\[
\left\langle \mathcal{J}^{\prime}(u),v\right\rangle =\int_{\mathbb{R}^{N}}%
\phi(\left\vert \nabla u\right\vert )\nabla u\cdot\nabla v+\int_{\mathbb{R}%
^{N}}V(x)\phi(\left\vert u\right\vert )uv-\int_{\mathbb{R}^{N}}f(u)v\qquad
u,v\in X\text{.}%
\]
Thus, critical points of $\mathcal{J}$ are precisely weak solutions of our
problem (\ref{e1}).

Under the assumptions $(\phi_{1})$, $(\phi_{2})$, $(V_{0})$, $(f_{1})$ and
$(f_{2})$, it has also been proved in \cite[Lemma 4.1]{MR3328350} that
$\mathcal{J}$ satisfies the mountain pass geometry: for some $\rho>0$ and
$\varphi\in C_{0}^{\infty}(\mathbb{R}^{N})\backslash\left\{  0\right\}  $,%
\begin{equation}
\inf_{\left\Vert u\right\Vert =\rho}\mathcal{J}(u)=\eta>0\text{,\qquad}%
\lim_{t\rightarrow+\infty}\mathcal{J}(t\varphi)=-\infty\text{.}
\label{e12}%
\end{equation}

\begin{remark}
\label{rr0}In \cite{MR3328350}, $\mathcal{J}(t\varphi)\rightarrow-\infty$ as
$t\rightarrow+\infty$ is only verified for $\varphi\in C_{0}^{\infty
}(\mathbb{R}^{N})\backslash\left\{  0\right\}  $. But we can prove that
$\mathcal{J}$ is anti-coercive on any finite dimensional subspace, see the
verifivation of condition (2) of Proposition \ref{p1} in the proof of Theorem
\ref{t1} (2) below. Therefore, the limit in (\ref{e12}) is in fact valid for any
$\varphi\in X$.
\end{remark}

Denote $I=\left[  0,1\right]  $ and set%
\begin{equation}
c=\inf_{\gamma\in\Gamma}\max_{t\in\left[  0,1\right]  }\mathcal{J}(\gamma(t))
\label{e}%
\end{equation}
being $\Gamma=\left\{  \left.  \gamma\in C(I,X)\right\vert \,\gamma
(0)=0\text{, }\mathcal{J}(\gamma(1))<0\right\}  $. Note that $c\geq\eta>0$.

According to the mountain pass theorem \cite{MR0370183,MR1127041}, there is a
sequence $\left\{  u_{n}\right\}  \subset X$ such that%
\begin{equation}
\mathcal{J}( u_{n}) \rightarrow c\text{,\qquad}\mathcal{J}^{\prime}( u_{n})
\rightarrow0\text{.} \label{e10}%
\end{equation}
Such sequence is called a $( PS) _{c}$ sequence (named after R. Palais and S.
Smale). Under the assumptions $( \phi_{1}) $, $( \phi_{2}) $, $( V_{0}) $, $(
f_{1}) $ and $( f_{2}) $, it has been shown in \cite[Lemma 4.2]{MR3328350}
that, the $( PS) _{c}$ sequence $\left\{  u_{n}\right\}  $ we just obtained is
bounded in $X$.

The following result has been established in \cite[Lemma 4.3]{MR3328350}.

\begin{lemma}
\label{l1}Suppose $( \phi_{1}) $, $( \phi_{2}) $, $( V_{0}) $ and $( f_{1}) $
hold. Let $\left\{  u_{n}\right\}  $ be a $( PS) _{c}$ sequence of
$\mathcal{J}$. If $u_{n}\rightharpoonup u$ in $X$, then $\nabla u_{n}%
\rightarrow\nabla u$ a.e.\ in $\mathbb{R}^{N}$ and $\mathcal{J}^{\prime}( u)
=0$.
\end{lemma}

\subsection{Proof of Theorem \ref{t1} (1)}

By the above arguments, we know that $\mathcal{J}$ has a bounded $( PS) _{c}$
sequence $\left\{  u_{n}\right\}  $. Since $X$ is reflexive, we may assume
that $u_{n}\rightharpoonup u$ in $X$. By Lemma \ref{l1}, $u$ is a critical
point of $\mathcal{J}$. We need to show that $u\neq0$. Thanks to the compact
embedding established in Lemma \ref{l0}, this can be achieved as in
\cite[p.\ 454]{MR3328350}. For the reader's convenience, we include the
argument below.

Assume that $u=0$. By Lemma \ref{l0}, the embedding $X\hookrightarrow L^{\Phi
}(\mathbb{R}^{N})$ is compact. Thus, $u_{n}\rightarrow0$ in $L^{\Phi
}(\mathbb{R}^{N})$ and we get%
\begin{equation}
\int_{\mathbb{R}^{N}}\Phi(\left\vert u_{n}\right\vert )\rightarrow0\text{.}
\label{e8}%
\end{equation}
By $(f_{1})$, for any $\varepsilon>0$, there exists $C_{\varepsilon}>0$ such
that%
\[
\left\vert f(t)t\right\vert \leq\varepsilon\Phi_{\ast}(\left\vert t\right\vert
)+C_{\varepsilon}\Phi(\left\vert t\right\vert )\text{.}%
\]
Using this inequality and (\ref{e8}), and the boundedness of $\left\{
u_{n}\right\}  $ in $L^{\Phi_{\ast}}(\mathbb{R}^{N})$, we deduce%
\[
\int_{\mathbb{R}^{N}}f(u_{n})u_{n}\rightarrow0\text{.}%
\]
Now, because $\left\langle \mathcal{J}^{\prime}(u_{n}),u_{n}\right\rangle
\rightarrow0$, we obtain%
\[
\int_{\mathbb{R}^{N}}\phi(\left\vert \nabla u_{n}\right\vert )\left\vert
\nabla u_{n}\right\vert ^{2}+\int_{\mathbb{R}^{N}}V(x)\phi(\left\vert
u_{n}\right\vert )u_{n}^{2}\rightarrow0\text{.}%
\]
From this and $(\phi_{2})$ we get
\[
\int_{\mathbb{R}^{N}}\Phi(\left\vert \nabla u_{n}\right\vert )+\int%
_{\mathbb{R}^{N}}V(x)\Phi(\left\vert u_{n}\right\vert )\rightarrow0\text{.}%
\]
That is $u_{n}\rightarrow0$ in $X$. But $\mathcal{J}(u_{n})\rightarrow c>0$,
this is a contradiction.

\begin{remark}
In Lemma \ref{lm4} we will show that $\mathcal{J}$ satisfies the $(PS)$ condition. Hence the $( PS) _{c}$
sequence $\left\{  u_{n}\right\}  $ has a subsequence converges to a nonzero critical point $u$ at the level $c>0$. We include the above argument for its simplicity.
\end{remark}

\subsection{Proof of Theorem \ref{t2}}

For convenience, in this subsection we assume all the conditions on $\phi$,
$V$ and $f$ required in Theorem \ref{t2}. As before, $\mathcal{J}$ has a
bounded $( PS) _{c}$ sequence $\left\{  u_{n}\right\}  $, $u_{n}%
\rightharpoonup u$ in $X=W^{1,\Phi}( \mathbb{R}^{N}) $ and $u$ is a critical
point of $\mathcal{J}$. To show that $u\neq0$, we need to consider the
limiting functional $\mathcal{J}_{\infty}:X\rightarrow\mathbb{R}$,%
\[
\mathcal{J}_{\infty}( u) =\int_{\mathbb{R}^{N}}\Phi( \left\vert \nabla
u\right\vert ) +\int_{\mathbb{R}^{N}}V_{\infty}\Phi( \left\vert u\right\vert )
-\int_{\mathbb{R}^{N}}F( u) \text{.}%
\]

\begin{lemma}
\label{l2}If $u=0$, then $\left\{  u_{n}\right\}  $ is also a $( PS) _{c}$
sequence of $\mathcal{J}_{\infty}$.
\end{lemma}

\begin{proof}
By condition $(V_{2})$, for any $\varepsilon>0$, there is $R>0$ such that%
\[
\left\vert V(x)-V_{\infty}\right\vert <\varepsilon\text{\qquad for }\left\vert
x\right\vert \geq R\text{.}%
\]
If $u=0$, then $u_{n}\rightharpoonup0$ in $X$. By the compactness of the
embedding $X\hookrightarrow L_{\mathrm{loc}}^{\Phi}(\mathbb{R}^{N})$ we have%
\[
\int_{\left\vert x\right\vert <R}\Phi(\left\vert u_{n}\right\vert
)\rightarrow0\text{.}%
\]
Consequently,
\begin{align*}
\left\vert \mathcal{J}_{\infty}(u_{n})-\mathcal{J}(u_{n})\right\vert  &
=\int_{\mathbb{R}^{N}}( V_{\infty}-V(x)) \Phi(\left\vert
u_{n}\right\vert )\\
&  =\left(  \int_{\left\vert x\right\vert <R}+\int_{\left\vert x\right\vert
\geq R}\right)  (V_{\infty}-V(x))\Phi(\left\vert u_{n}\right\vert )\\
&  \leq V_{\infty}\int_{\left\vert x\right\vert <R}\Phi(\left\vert
u_{n}\right\vert )+\varepsilon\int_{\left\vert x\right\vert \geq R}%
\Phi(\left\vert u_{n}\right\vert )\\
&  \leq V_{\infty}\int_{\left\vert x\right\vert <R}\Phi(\left\vert
u_{n}\right\vert )+C\varepsilon\text{,}%
\end{align*}
because $\{u_{n}\}$ is bounded in $L^{\Phi}(\mathbb{R}^{N})$. It follows that
\[
\varlimsup_{n\rightarrow\infty}\left\vert \mathcal{J}_{\infty}(u_{n}%
)-\mathcal{J}(u_{n})\right\vert \leq C\varepsilon\text{,}%
\]
which implies $\mathcal{J}_{\infty}(u_{n})-\mathcal{J}(u_{n})\rightarrow0$.

In a similar manner we can prove%
\[
\left\Vert \mathcal{J}_{\infty}^{\prime}( u_{n}) -\mathcal{J}^{\prime}( u_{n})
\right\Vert =\sup_{h\in X,\left\Vert h\right\Vert =1}\left\vert \int%
_{\mathbb{R}^{N}}( V_{\infty}-V( x) ) u_{n}h\right\vert \rightarrow0\text{.}%
\]
Thus $\mathcal{J}_{\infty}( u_{n}) \rightarrow c$ and $\mathcal{J}_{\infty
}^{\prime}( u_{n}) \rightarrow0$.
\end{proof}

Considering the constant potential $V_{\infty}$ as a $\mathbb{Z}^{N}$-periodic
function, it has been shown in the proof of \cite[Theorem 1.8 (b)]{MR3328350}
that for $\left\{  u_{n}\right\}  $, the bounded $( PS) _{c}$ sequence of
$\mathcal{J}_{\infty}$ obtained in Lemma \ref{l2}, there exists a sequence
$\left\{  y_{n}\right\}  \subset\mathbb{R}^{N}$ such that setting $v_{n}( x)
=u_{n}(x-y_{n})$ for $x\in\mathbb{R}^{N}$, then $v_{n}\rightharpoonup v$ in
$X$, and $v$ is a nonzero critical point of $\mathcal{J}_{\infty}$.

We claim that $\mathcal{J}_{\infty}(v)\leq c$. In addition to the
obvious fact that $v_{n}\rightarrow v$ a.e.\ in $\mathbb{R}^{N}$, by applying
Lemma \ref{l1} to $\mathcal{J}_{\infty}$ we also have $\nabla v_{n}%
\rightarrow\nabla v$ a.e.\ in $\mathbb{R}^{N}$. By the assumptions $(\phi_{2}%
)$, $(f_{2})$, and $\theta>m$, we have
\begin{equation}
\Phi(\left\vert t\right\vert )-\frac{1}{\theta}\phi(\left\vert t\right\vert
)t^{2}\geq\left(  1-\frac{m}{\theta}\right)  \Phi(\left\vert t\right\vert
)\geq0\text{,\qquad}\frac{1}{\theta}f(t)t-F(t)\geq0 \label{e99}%
\end{equation}
for $t\ge0$. Hence, we may apply the Fatou lemma to get%
\begin{align}
c  &  =\lim_{n\rightarrow\infty}\left\{  \mathcal{J}_{\infty}(v_{n})-\frac
{1}{\theta}\left\langle \mathcal{J}_{\infty}^{\prime}(v_{n}),v_{n}%
\right\rangle \right\} \nonumber\\
&  =\varliminf_{n\rightarrow\infty}\left\{  \int_{\mathbb{R}^{N}}\left(
\Phi(\left\vert \nabla v_{n}\right\vert )-\frac{1}{\theta}\phi(\left\vert
\nabla v_{n}\right\vert )\left\vert \nabla v_{n}\right\vert ^{2}\right)
\right. \nonumber\\
&  \qquad\qquad\left.  +\int_{\mathbb{R}^{N}}V_\infty\left(  \Phi(\left\vert
v_{n}\right\vert )-\frac{1}{\theta}\phi(\left\vert v_{n}\right\vert
)\left\vert v_{n}\right\vert ^{2}\right)  +\int_{\mathbb{R}^{N}}\left(
\frac{1}{\theta}f(v_{n})v_{n}-F(v_{n})\right)  \right\} \nonumber\\
&  \geq\int_{\mathbb{R}^{N}}\left(  \Phi(\left\vert \nabla v\right\vert
)-\frac{1}{\theta}\phi(\left\vert \nabla v\right\vert )\left\vert \nabla
v\right\vert ^{2}\right) \nonumber\\
&  \qquad\qquad+\int_{\mathbb{R}^{N}}V_\infty\left(  \Phi(\left\vert v\right\vert
)-\frac{1}{\theta}\phi(\left\vert v\right\vert )\left\vert v\right\vert
^{2}\right)  +\int_{\mathbb{R}^{N}}\left(  \frac{1}{\theta}f(v)v-F(v)\right)
\nonumber\\
&  =\mathcal{J}_{\infty}(v)-\frac{1}{\theta}\left\langle \mathcal{J}_{\infty
}^{\prime}(v),v\right\rangle =\mathcal{J}_{\infty}(v)\text{.} \label{e9}%
\end{align}

Define a $C^{1}$-function $\varrho:[0,\infty)\rightarrow\mathbb{R}$ by%
\[
\varrho(t)=\mathcal{J}_{\infty}(tv)=\int_{\mathbb{R}^{N}}\Phi(t\left\vert
\nabla v\right\vert )+\int_{\mathbb{R}^{N}}V_{\infty}\Phi(t\left\vert
v\right\vert )-\int_{\mathbb{R}^{N}}F(tv)\text{.}%
\]
Then for $t>0$,%
\begin{align*}
\varrho^{\prime}(t)  &  =\left\langle \mathcal{J}_{\infty}^{\prime
}(tv),v\right\rangle \\
&  =t\int_{\mathbb{R}^{N}}\left(  \phi(t\left\vert \nabla v\right\vert
)\left\vert \nabla v\right\vert ^{2}+V_{\infty}\phi(t\left\vert v\right\vert
)v^{2}\right)  -\int_{\mathbb{R}^{N}}f(tv)v\text{.}%
\end{align*}
Hence, for the $s\geq2$ in assumptions $(\phi_{3}^{s})$ and $(f_{3}^{s})$ we
have
\begin{align*}
\varrho^{\prime}(t)  &  >0\quad\Longleftrightarrow\quad\frac{1}{t^{s-2}}\int%
_{\mathbb{R}^{N}}\left(  \phi(t\left\vert \nabla u\right\vert )\left\vert
\nabla u\right\vert ^{2}+V_{\infty}\phi(t\left\vert v\right\vert
)v^{2}\right)  >\int_{\mathbb{R}^{N}}\frac{f(tv)v}{t^{s-1}}\text{,}\\
\varrho^{\prime}(t)  &  =0\quad\Longleftrightarrow\quad\frac{1}{t^{s-2}}\int%
_{\mathbb{R}^{N}}\left(  \phi(t\left\vert \nabla u\right\vert )\left\vert
\nabla u\right\vert ^{2}+V_{\infty}\phi(t\left\vert v\right\vert
)v^{2}\right)  =\int_{\mathbb{R}^{N}}\frac{f(tv)v}{t^{s-1}}\text{,}\\
\varrho^{\prime}(t)  &  <0\quad\Longleftrightarrow\quad\frac{1}{t^{s-2}}\int%
_{\mathbb{R}^{N}}\left(  \phi(t\left\vert \nabla u\right\vert )\left\vert
\nabla u\right\vert ^{2}+V_{\infty}\phi(t\left\vert v\right\vert
)v^{2}\right)  <\int_{\mathbb{R}^{N}}\frac{f(tv)v}{t^{s-1}}\text{.}%
\end{align*}
Since $\varrho^{\prime}(1)=\left\langle \mathcal{J}_{\infty}^{\prime
}(v),v\right\rangle =0$, by $(\phi_{3}^{s})$ and $(f_{3}^{s})$ and a
monotonicity argument we see that
\[
\varrho^{\prime}(t)>0\text{\quad\ for }t\in(0,1)\text{,\qquad}\varrho^{\prime
}(t)<0\text{\quad\ for }t\in(1,\infty)\text{.}%
\]
Hence,%
\begin{equation}
\mathcal{J}_{\infty}(v)=\varrho(1)=\max_{t\geq0}\varrho(t)=\max_{t\geq
0}\mathcal{J}_{\infty}(tv)\text{.} \label{e11}%
\end{equation}

Now, we are ready to conclude the proof of Theorem \ref{t2}. For the bounded
$( PS) _{c}$ sequence $\left\{  u_{n}\right\}  $ given in (\ref{e10}), we
known that the weak limit $u$ of a subsequence is a critical point of
$\mathcal{J}$. If $u=0$, by Lemma \ref{l2}, this $\left\{  u_{n}\right\}  $ is
also a $( PS) _{c}$ sequence of the limiting functional $\mathcal{J}_{\infty}%
$, which will produce a nonzero critical point $v$ of $\mathcal{J}_{\infty}$
satisfying $\mathcal{J}_{\infty}( v) \leq c$, see (\ref{e9}).

We also know that $\mathcal{J}_{\infty}(tv)\rightarrow-\infty$ as
$t\rightarrow+\infty$, see Remark \ref{rr0}. Choose $T>0$ such that
$\mathcal{J}_{\infty}(Tv)<0$ and define $\gamma:\left[  0,1\right]  \rightarrow
X$, $\gamma(t)=tTv$. Then $\gamma(0)=0$,
\[
\mathcal{J}(\gamma(1))<\mathcal{J}_{\infty}(\gamma(1))=\mathcal{J}_{\infty
}(Tv)<0
\]
because $V(x)<V_\infty$ for all $x\in\mathbb{R}^N$, hence $\gamma\in\Gamma$.

By assumption $(V_{2})$, (\ref{e9}) and (\ref{e11}) we see that for $t\in(0,1]$,%
\begin{equation}
\mathcal{J}(\gamma(t))<\mathcal{J}_{\infty}(\gamma(t))\leq\mathcal{J}_{\infty
}(v)\leq c\text{.} \label{ew0}%
\end{equation}
Because $\gamma(0)=0$, $\mathcal{J}(0)=0$ and $c>0$, (\ref{ew0}) is also true
at $t=0$. Hence
\[
\max_{t\in\left[  0,1\right]  }\mathcal{J}(\gamma(t))=\mathcal{J}(\gamma(t_0))<\mathcal{J}_{\infty
}(v)\leq c
\]
for some $t_0\in[0,1]$, contradicting the definition of $c$ given in (\ref{e}). Therefore $u\neq0$
and it is a nonzero critical point of $\mathcal{J}$.

\subsection{Proof of Theorem \ref{t3}}

On the subspace
\[
E=\left\{  u\in W^{1,\Phi}(\mathbb{R}^{N})\left\vert \,\int_{\mathbb{R}^{N}%
}a(x)\Phi(\left\vert u\right\vert )<\infty\right.  \right\} 
\]
of $W^{1,\Phi}(\mathbb{R}^{N})$, we equip the norm%
\[
\left\Vert u\right\Vert =\left\vert \nabla u\right\vert _{\Phi}+\left\vert
u\right\vert _{\Phi,(a+1)}\text{.}%
\]
Then $E$ becomes a Banach space. To prove Theorem \ref{t3} we only need to
find nonzero critical points of $\mathcal{J}_{\lambda}:E\rightarrow\mathbb{R}%
$,%
\[
\mathcal{J}_{\lambda}(u)=\int_{\mathbb{R}^{N}}\Phi(\left\vert \nabla
u\right\vert )+\int_{\mathbb{R}^{N}}(\lambda a(x)+1)\Phi(\left\vert
u\right\vert )-\int_{\mathbb{R}^{N}}F(u)\text{.}%
\]
As before, $\mathcal{J}_{\lambda}$ verifies the assumptions of the mountain pass theorem thus has
a $(PS)_{c_{\lambda}}$ sequence $\big\{u_{n}^{\lambda}\big\}_{n=1}^\infty$ satisfying%
\begin{equation}
\mathcal{J}_{\lambda}(u_{n}^{\lambda})\rightarrow c_{\lambda}>0\text{,\qquad
}\mathcal{J}_{\lambda}^{\prime}(u_{n}^{\lambda})\rightarrow0\text{,}
\label{e13}%
\end{equation}
where%
\begin{equation}
c_{\lambda}=\inf_{\gamma\in\Gamma_{\lambda}}\max_{t\in\left[  0,1\right]
}\mathcal{J}_{\lambda}(\gamma(t)) \label{ep4}%
\end{equation}
being $\Gamma_{\lambda}=\left\{  \left.  \gamma\in C(I,E)\right\vert
\,\gamma(0)=0\text{, }\mathcal{J}_{\lambda}(\gamma(1))<0\right\}  $.

Moreover, the sequence $\big\{u_{n}^{\lambda}\big\}_{n=1}^\infty$ is bounded in $E$. Going to
a subsequence if necessary, $u_{n}^{\lambda}\rightharpoonup u^{\lambda}$ in $E$, and
$u^{\lambda}$ is a critical point of $\mathcal{J}_{\lambda}$ due to Lemma \ref{l1}. We need to show
that $u^{\lambda}\neq0$. Although the basic idea can be traced back to
\cite[\S 5]{MR1349229}, we need to create the required estimates more
carefully because our differential operator $-\Delta_{\Phi}$ is much more
complicated than in \cite{MR1349229}.

Take an $\mathcal{N}$-function satisfying (\ref{e2}). From $\left(
f_{1}\right)  $ with $\lambda_{0}=0$, for any $\varepsilon>0$, there exists
$C_{\varepsilon}>0$ such that%
\begin{equation}
\frac{1}{\ell}f(t)t-F(t)\leq\varepsilon\left(  \Phi(\left\vert t\right\vert)
+\Phi_{\ast}(\left\vert t\right\vert )\right)  +C_{\varepsilon}\Psi(\left\vert
t\right\vert )\text{.}\label{ufo}%
\end{equation}
By assumption $(a_{1})$, we can take $v\in E\backslash\left\{  0\right\}  $,
such that $\operatorname*{supp}v$ is contained in the interior of $a^{-1}(0)$.
Then, by the mountain pass characterization of $c_{\lambda}$ in (\ref{ep4}),
we have
\begin{align}
c_{\lambda} &  \leq\max_{t\geq0}\mathcal{J}_{\lambda}(tv)\nonumber\\
&  =\max_{t\geq0}\left\{  \int\Phi(t\left\vert \nabla v\right\vert )+\int
\Phi(t\left\vert v\right\vert )-\int F(tv)\right\}  =\tilde{c}<+\infty
\text{,}\label{e17}%
\end{align}
see the proof of \cite[Lemma 4.1 (b)]{MR3328350} for more details.

\begin{lemma}
\label{l3}There exists $\alpha>0$ such that for all $\lambda\geq1$,%
\[
\varliminf_{n\rightarrow\infty}\int_{\mathbb{R}^{N}}\Psi( \vert u_{n}%
^{\lambda}\vert) \geq\alpha\text{.}%
\]

\end{lemma}

\begin{proof}
It has been shown in \cite[Lemma 4.1 (a)]{MR3328350} that $u=0$ is a strict
local minimizer of $\mathcal{J}_{0}$. Since for $\lambda>0$ we have
$\mathcal{J}_{\lambda}\geq\mathcal{J}_{0}$, it follows that $\Gamma_{\lambda
}\subset\Gamma_{0}$. By (\ref{ep4}), it is easy to see that $c_{\lambda}\geq
c_{0}>0$.

For simplicity of notation, in the proof of Lemmas \ref{l3} and \ref{l4} we drop the superscript $\lambda$ and
write $u_{n}$ for $u_{n}^{\lambda}$. From (\ref{e13}), using (\ref{e99}) we have (note that $\lambda\geq1$)%
\begin{align}
c_{\lambda} &  =\lim_{n\rightarrow\infty}\left\{  \mathcal{J}_{\lambda}%
(u_{n})-\frac{1}{\theta}\left\langle \mathcal{J}_{\lambda}^{\prime}%
(u_{n}),u_{n}\right\rangle \right\}  \nonumber\\
&  \geq\varlimsup_{n\rightarrow\infty}\left\{  \int_{\mathbb{R}^{N}}\left(
\Phi(\left\vert \nabla u_{n}\right\vert )-\frac{1}{\theta}\phi(\left\vert
\nabla u_{n}\right\vert )\left\vert \nabla u_{n}\right\vert ^{2}\right)
\right.  \nonumber\\
&  \qquad\qquad\left.  +\int_{\mathbb{R}^{N}}(\lambda a(x)+1)\left(
\Phi(\left\vert u_{n}\right\vert )-\frac{1}{\theta}\phi(\left\vert
u_{n}\right\vert )\left\vert u_{n}\right\vert ^{2}\right)  \right\}
\nonumber\\
&  \geq\left(  1-\frac{m}{\theta}\right)  \varlimsup_{n\rightarrow\infty
}\left\{  \int_{\mathbb{R}^{N}}\Phi(\left\vert \nabla u_{n}\right\vert
)+\int_{\mathbb{R}^{N}}(\lambda a(x)+1)\Phi(\left\vert u_{n}\right\vert
)\right\}  \text{.}\label{uof}%
\end{align}
Since the first integral in the last line is nonnegative, it follows that%
\begin{equation}
\varlimsup_{n\rightarrow\infty}\int_{\mathbb{R}^{N}}(\lambda a(x)+1)\Phi
(\left\vert u_{n}\right\vert )\leq\frac{\theta m}{\theta-m}c_{\lambda}%
\text{.}\label{e16}%
\end{equation}
Moreover, as indicated in (\ref{e17}), $\left\{  c_{\lambda}\right\}
_{\lambda\geq1}$ is bounded above by $\tilde{c}$, it follows from (\ref{uof})
that $\left\{  u_{n}\right\}  $ is bounded in $E$ by a positive constant, which is
independent of $\lambda$. Therefore, by the continuous embedding $E\hookrightarrow L^{\Phi_*}$, there exists $d>0$ such that%
\begin{equation}
\int_{\mathbb{R}^{N}}\Phi_{\ast}(\left\vert u_{n}\right\vert )\leq
d\text{.}\label{eud}%
\end{equation}

Using $(\phi_{2})$ we have $\Phi(t)\leq\ell^{-1}\phi(t)t^{2}$ for $t\geq0$,
then using (\ref{ufo}), (\ref{eud}) and (\ref{e16}) we get
\begin{align}
c_{\lambda} &  =\lim_{n\rightarrow\infty}\left\{  \mathcal{J}_{\lambda}%
(u_{n})-\frac{1}{\ell}\left\langle \mathcal{J}_{\lambda}^{\prime}(u_{n}%
),u_{n}\right\rangle \right\}  \nonumber\\
&  =\lim_{n\rightarrow\infty}\left\{  \int_{\mathbb{R}^{N}}\left(
\Phi(\left\vert \nabla u_{n}\right\vert )-\frac{1}{\ell}\phi(\left\vert \nabla
u_{n}\right\vert )\left\vert \nabla u_{n}\right\vert ^{2}\right)  \right.
\nonumber\\
&  \qquad\quad\left.  +\int_{\mathbb{R}^{N}}(\lambda a(x)+1)\left(
\Phi(\left\vert u_{n}\right\vert )-\frac{1}{\ell}\phi(\left\vert
u_{n}\right\vert )u_{n}^{2}\right)  +\int_{\mathbb{R}^{N}}\left(  \frac
{1}{\ell}f(u_{n})u_{n}-F(u_{n})\right)  \right\}  \nonumber\\
&  \leq\varliminf_{n\rightarrow\infty}\int_{\mathbb{R}^{N}}\left(  \frac
{1}{\ell}f(u_{n})u_{n}-F(u_{n})\right)  \nonumber\\
&  \leq\varliminf_{n\rightarrow\infty}\left(  \varepsilon\left(
\int_{\mathbb{R}^{N}}\Phi(\left\vert u_{n}\right\vert )+\int_{\mathbb{R}^{N}%
}\Phi_{\ast}(\left\vert u_{n}\right\vert )\right)  +C_{\varepsilon}%
\int_{\mathbb{R}^{N}}\Psi(\left\vert u_{n}\right\vert )\right)  \nonumber\\
&  \leq\varepsilon\varlimsup_{n\rightarrow\infty}\left(  d+\int_{\mathbb{R}%
^{N}}(\lambda a(x)+1)\Phi(\left\vert u_{n}\right\vert )\right)  +C_{\varepsilon}%
\varliminf_{n\rightarrow\infty}\int_{\mathbb{R}^{N}}\Psi(\left\vert
u_{n}\right\vert )\nonumber\\
&  \leq\varepsilon d+\frac{\theta m}{\theta-m}c_{\lambda}\varepsilon
+C_{\varepsilon}\varliminf_{n\rightarrow\infty}\int_{\mathbb{R}^{N}}%
\Psi(\left\vert u_{n}\right\vert )\text{.}\label{e18}%
\end{align}
Noting that $c_{\lambda}\geq c_{0}>0$, choosing $\varepsilon$ small enough at
the very beginning, the desired conclusion follows from (\ref{e18}).
\end{proof}

\begin{remark}\label{r26}
Comparing with the argument in \cite{MR1349229} for the case $\phi(t)\equiv1$, because instead of being \emph{strictly subcritical} (meaning that $f(t)$ is controlled by some subcritical $\mathcal{N}$-function $\Psi$), our nonlinearity $f(t)$ is only \emph{asymptotically subcritical}, in the estimate (\ref{e18}) the term involving $\Phi_*$ presents. Hence we need the uniform bound (\ref{eud}), which in turn is ensured by the upper bound of $\{c_\lambda\}$ given in (\ref{e17}).
\end{remark}

\begin{lemma}
\label{l4}For any $\varepsilon>0$, there exists $\lambda^{\ast}\ge1$ and $R>0$,
such that for $\lambda\geq\lambda^{\ast}$ we have%
\[
\varlimsup_{n\rightarrow\infty}\int_{\left\vert x\right\vert \geq R}\Psi(
\vert u_{n}^{\lambda}\vert) <\varepsilon\text{.}%
\]

\end{lemma}

\begin{proof}
For $R>0$, we set%
\begin{align*}
A_{R}  &  =\left\{  \left.  x\in\mathbb{R}^{N}\right\vert \,\left\vert
x\right\vert \geq R,a(x)\geq M_{0}\right\}  \text{,}\\
B_{R}  &  =\left\{  \left.  x\in\mathbb{R}^{N}\right\vert \,\left\vert
x\right\vert \geq R,a(x)<M_{0}\right\}  \text{.}%
\end{align*}
As in (\ref{ex}), because of (\ref{e2}) there exists $k>0$ such that%
\begin{equation}
\int_{\left\vert u_{n}\right\vert >k}\Psi(\left\vert u_{n}\right\vert
)\leq\frac{\varepsilon}{4}\text{.} \label{e19}%
\end{equation}
Using assumption $(a_{2})$, as $R\rightarrow\infty$ we have $\mu
(B_{R})\rightarrow0$, therefore we can fix $R>0$ such that%
\begin{equation}
\int_{B_{R}\cap\left\{  \left\vert u_{n}\right\vert \leq k\right\}  }%
\Psi(\left\vert u_{n}\right\vert )\leq\mu(B_{R})\cdot\sup_{t\in\left[  0,k\right]  }%
\Psi(t)<\frac{\varepsilon}{4}\text{.} \label{e20}%
\end{equation}
By the first limit from (\ref{e2}), there is $\kappa>0$ such that $\Psi(t)\leq\kappa\Phi(t)$ for
$t\in\left[  0,k\right]  $. Thus using
(\ref{e16}) and (\ref{e17}) we see that if $\lambda$ is large enough,
\begin{align}
\int_{A_{R}\cap\left\{  \left\vert u_{n}\right\vert \leq k\right\}  }%
\Psi(\left\vert u_{n}\right\vert )  &  \leq\kappa\int_{A_{R}}\Phi(\left\vert
u_{n}\right\vert )\nonumber\\
&  \leq\frac{\kappa}{\lambda M_{0}+1}\int_{A_{R}}(\lambda a(x)+1)\Phi
(\left\vert u_{n}\right\vert )\nonumber\\
&  \leq\frac{\kappa}{\lambda M_{0}+1}\frac{2\theta m}{\theta-m}c_{\lambda}%
\leq\frac{\kappa}{\lambda M_{0}+1}\frac{2\theta m}{\theta-m}\tilde{c}%
<\frac{\varepsilon}{4}\text{.} \label{e21}%
\end{align}
For such large $\lambda$, combining (\ref{e19}), (\ref{e20}) and (\ref{e21}),
we see that for the chosen $R>0$,%
\[
\varlimsup_{n\rightarrow\infty}\int_{\left\vert x\right\vert \geq R}%
\Psi(|u_{n}|)\leq\varlimsup_{n\rightarrow\infty}\left(  \int_{\left\vert
u_{n}\right\vert >k}+\int_{B_{R}\cap\left\{  \left\vert u_{n}\right\vert \leq
k\right\}  }+\int_{A_{R}\cap\left\{  \left\vert u_{n}\right\vert \leq
k\right\}  }\right)  \Psi(\left\vert u_{n}\right\vert )<\varepsilon\text{.}
\]

\end{proof}

Having proven Lemmas \ref{l3} and \ref{l4}, we are ready to complete the proof
of Theorem \ref{t3}. Set $\varepsilon=\alpha/2$ in Lemma \ref{l4} and fix
$\lambda^{\ast}>0$ and $R>0$ as in the lemma. If $\lambda\geq\lambda^{\ast}$,
for $\big\{u_{n}^{\lambda}\big\}$, the $(PS)_{c_{\lambda}}$ sequence of
$\mathcal{J}_{\lambda}$, we have $u_{n}^{\lambda}\rightharpoonup u^{\lambda}$
and $u^{\lambda}$ is a critical point of $\mathcal{J}_{\lambda}$. Since the
embedding $E\hookrightarrow L^{\Psi}(B_{R})$ is compact,%
\begin{align*}
\int_{\left\vert x\right\vert <R}\Psi(|u^{\lambda}|)  &  =\lim_{n\rightarrow
\infty}\int_{\left\vert x\right\vert <R}\Psi(|u_{n}^{\lambda}|)\\
&  \geq\varliminf_{n\rightarrow\infty}\int_{\mathbb{R}^{N}}\Psi(|u_{n}%
^{\lambda}|)-\varlimsup_{n\rightarrow\infty}\int_{\left\vert x\right\vert \geq
R}\Psi(|u_{n}^{\lambda}|)\geq\frac{\varepsilon}{2}\text{.}%
\end{align*}
Therefore, $u^{\lambda}$ is a nonzero critical point of $\mathcal{J}_{\lambda
}$.

\section{Multiple solutions}

\begin{lemma}\label{lm4}
Under the assumptions of Theorem \ref{t1}, $\mathcal{J}$ satisfies the $( PS)
$ condition, that is, for any $c\in\mathbb{R}$, all $( PS) _{c}$ sequence of
$\mathcal{J}$ has convergent subsequence.
\end{lemma}

\begin{proof}
Let $\left\{  u_{n}\right\}  $ be a $(PS)_{c}$ sequence of $\mathcal{J}$. Then
$\left\{  u_{n}\right\}  $ is bounded and we may assume that $u_{n}%
\rightharpoonup u$ in $X$. Firstly we show that up to a subsequence%
\begin{equation}
\int_{\mathbb{R}^{N}}f(u_{n})(u_{n}-u)\rightarrow0\text{.} \label{e22}%
\end{equation}
By assumption $(f_{1})$, for any $\varepsilon>0$, there is $C_{\varepsilon}>0$
such that%
\begin{align}
\left\vert f(t)\right\vert  &  \leq\varepsilon\phi_{\ast}(\left\vert
t\right\vert )\left\vert t\right\vert +C_{\varepsilon}\phi(\left\vert
t\right\vert )\left\vert t\right\vert \nonumber\\
&  =\varepsilon\Phi_{\ast}^{\prime}(\left\vert t\right\vert )+C_{\varepsilon
}\Phi^{\prime}(\left\vert t\right\vert )\text{.} \label{e24}%
\end{align}
For $u\in X$, by H\"{o}lder inequality (\ref{eh}) we have%
\[
\int_{\mathbb{R}^{N}}\Phi^{\prime}(\left\vert u\right\vert )\left\vert
v\right\vert \leq2\vert \Phi^{\prime}(\vert u\vert
)\vert _{\tilde{\Phi}}\left\vert v\right\vert _{\Phi}\text{.}%
\]
Note that from (\ref{ep1}), (\ref{e6}) and Lemma \ref{l5} with $V\equiv1$, we
have%
\[
\int_{\mathbb{R}^{N}}\tilde{\Phi}(\vert \Phi^{\prime}(\left\vert
u\right\vert )\vert )\leq\int_{\mathbb{R}^{N}}\Phi(2\left\vert
u\right\vert )\leq\xi_{1}(2)\int_{\mathbb{R}^{N}}\Phi(\left\vert u\right\vert
)\leq2^{m}\xi_{1}(\left\vert u\right\vert _{\Phi})\text{.}%
\]
Therefore, since $\left\{  u_{n}\right\}  $ is bounded in $L^{\Phi}%
(\mathbb{R}^{N})$, $\left\{  \Phi^{\prime}(\left\vert u_{n}\right\vert
)\right\}  $ is also bounded in $L^{\tilde{\Phi}}(\mathbb{R}^{N})$. Similarly,
$\left\{  \Phi_{\ast}^{\prime}(\left\vert u_{n}\right\vert )\right\}  $ is
bounded in $L^{\widetilde{\Phi_{\ast}}}(\mathbb{R}^{N})$. (We remind the reader that instead of the Sob\-olev conjugate function of $\tilde{\Phi}$, here $\widetilde{\Phi_{\ast}}$ is the  complement function of $\Phi_{\ast}$.) Therefore%
\begin{equation}
M:=2\sup_{n}\vert \Phi_{\ast}^{\prime}(\vert u_{n}\vert
)\vert _{\widetilde{\Phi_{\ast}}}\left\vert u_{n}-u\right\vert
_{\Phi_{\ast}}<+\infty\text{.} \label{eex}%
\end{equation}

Because $u_{n}\rightharpoonup u$ in $X$, by Lemma \ref{l0} we have
$u_{n}\rightarrow u$ in $L^{\Phi}(\mathbb{R}^{N})$. Now, using (\ref{e24}) and
H\"{o}lder inequality we get%
\begin{align*}
\left\vert \int_{\mathbb{R}^{N}}f(u_{n})(u_{n}-u)\right\vert  &
\leq\varepsilon\int_{\mathbb{R}^{N}}\Phi_{\ast}^{\prime}(\left\vert u_{n}\right\vert
)\left\vert u_{n}-u\right\vert +C_{\varepsilon}\int_{\mathbb{R}^{N}}\Phi^{\prime}(\left\vert
u_{n}\right\vert )\left\vert u_{n}-u\right\vert \\
&  \leq2\varepsilon\,\vert \Phi_{\ast}^{\prime}(\left\vert u_{n}\right\vert
)\vert _{\widetilde{\Phi_{\ast}}}\left\vert u_{n}-u\right\vert
_{\Phi_{\ast}}+2C_{\varepsilon}\,\vert \Phi^{\prime}(\left\vert
u_{n}\right\vert )\vert _{\tilde{\Phi}}\left\vert u_{n}-u\right\vert
_{\Phi}\text{.}%
\end{align*}
Since $u_{n}\rightarrow u$ in $L^{\Phi}(\mathbb{R}^{N})$, using (\ref{eex}) and the boundedness of $\left\{  \Phi^{\prime}(\left\vert u_{n}\right\vert
)\right\}  $ in $L^{\tilde{\Phi}}(\mathbb{R}^{N})$, it
follows that%
\[
\varlimsup_{n\rightarrow\infty}\left\vert \int_{\mathbb{R}^{N}}f(u_{n}%
)(u_{n}-u)\right\vert \leq M\varepsilon\text{,}%
\]
which implies (\ref{e22}).

To prove that
$u_{n}\rightarrow u$ in $X$, we adapt the argument of \cite[Appendix A]{MR3306384}, where for $V(x)\equiv0$, a $\Phi
$-Laplacian problem on a bounded domain is considered. Let $\mathcal{A}:X\rightarrow X^{\ast}$ be
defined by%
\[
\left\langle \mathcal{A}(u),v\right\rangle =\int_{\mathbb{R}^{N}}%
\phi(\left\vert \nabla u\right\vert )\nabla u\cdot\nabla v+\int_{\mathbb{R}%
^{N}}V(x)\phi(\left\vert u\right\vert )uv\text{.}%
\]
Then it is well known that

\begin{itemize}
\item $\mathcal{A}$ is hemicontinuous, i.e., for all $u,v,w\in X$, the
function%
\[
t\mapsto\left\langle \mathcal{A}( u+tv) ,w\right\rangle
\]
is continuous on $\left[  0,1\right]  $.

\item $\mathcal{A}$ is strictly monotone: $\left\langle \mathcal{A}( u)
-\mathcal{A}( v) ,u-v\right\rangle >0$ for $u,v\in X$ with $u\neq v$.
\end{itemize}
By \cite[Lemma 2.98]{MR2267795}, we know that $\mathcal{A}$ is pseudomonotone,
i.e., for $\left\{  u_{n}\right\}  \subset X$,%
\begin{equation}
u_{n}\rightharpoonup u\text{\quad in }X\text{,}\qquad\qquad\varlimsup\limits_{n\rightarrow\infty}\left\langle \mathcal{A}( u_{n})
,u_{n}-u\right\rangle \leq0 \label{e25}%
\end{equation}
together imply $\mathcal{A}( u_{n}) \rightharpoonup\mathcal{A}( u) $ in $X^{\ast}$
and $\left\langle \mathcal{A}( u_{n}) ,u_{n}\right\rangle \rightarrow
\left\langle \mathcal{A}u,u\right\rangle $.

For our bounded $(PS)_{c}$ sequence $\left\{  u_{n}\right\}  $, (\ref{e22})
implies that (\ref{e25}) holds up to a subsequence. Therefore%
\begin{equation}
\left\langle \mathcal{A}(u_{n}),u_{n}\right\rangle \rightarrow\left\langle
\mathcal{A}u,u\right\rangle \text{.} \label{e26}%
\end{equation}
According to Lemma \ref{l1}, in addition to the well known $u_{n}\rightarrow
u$ a.e.\ in $\mathbb{R}^{N}$ we also have $\nabla u_{n}\rightarrow\nabla u$
a.e.\ in $\mathbb{R}^{N}$. By the continuity of $\Phi$ we get%
\begin{equation}
f_{n}:=\Phi(\left\vert \nabla u_{n}-\nabla u\right\vert )+V(x)\Phi(\left\vert
u_{n}-u\right\vert )\rightarrow0\text{\qquad a.e.\ in }\mathbb{R}^{N}\text{.}
\label{e28}%
\end{equation}
Let $g_{n}:\mathbb{R}^{N}\rightarrow\mathbb{R}$ be defined by%
\[
g_{n}=\frac{2^{m-1}}{\ell}\left\{  \phi(\left\vert \nabla u_{n}\right\vert
)\left\vert \nabla u_{n}\right\vert ^{2}+V(x)\phi(\left\vert u_{n}\right\vert
)\left\vert u_{n}\right\vert ^{2}+\Phi(\left\vert \nabla u\right\vert
)+V(x)\Phi(\left\vert u\right\vert )\right\}  \text{.}%
\]
Then by the monotonicity and convexity of $\Phi$, using (\ref{e6}) and
$(\phi_{2})$ we get%
\begin{align*}
\left\vert f_{n}\right\vert  &  \leq\Phi\left(  \frac{2\left\vert \nabla
u_{n}\right\vert +2\left\vert \nabla u\right\vert }{2}\right)  +V(x)\Phi
\left(  \frac{2\left\vert u_{n}\right\vert +2\left\vert u\right\vert }%
{2}\right) \\
&  \leq\frac{\Phi(2\left\vert \nabla u_{n}\right\vert )+\Phi(2\left\vert
\nabla u\right\vert )}{2}+V(x)\frac{\Phi(2\left\vert u_{n}\right\vert
)+\Phi(2\left\vert u\right\vert )}{2}\\
&  \leq2^{m-1}\left\{  \left[  \Phi(\left\vert \nabla u_{n}\right\vert
)+\Phi(\left\vert \nabla u\right\vert )\right]  +V(x)\left[  \Phi(\left\vert
u_{n}\right\vert )+\Phi(\left\vert u\right\vert )\right]  \right\} \\
&  \leq\frac{2^{m-1}}{\ell}\left\{  \phi(\left\vert \nabla u_{n}\right\vert
)\left\vert \nabla u_{n}\right\vert ^{2}+V(x)\phi(\left\vert u_{n}\right\vert
)\left\vert u_{n}\right\vert ^{2}+\Phi(\left\vert \nabla u\right\vert
)+V(x)\Phi(\left\vert u\right\vert )\right\}  =g_{n}\text{.}%
\end{align*}
We have%
\begin{equation}
g_{n}\rightarrow g:=\frac{2^{m-1}}{\ell}\left\{  \phi(\left\vert \nabla
u\right\vert )\left\vert \nabla u\right\vert ^{2}+V(x)\phi(\left\vert
u\right\vert )\left\vert u\right\vert ^{2}+\Phi(\left\vert \nabla u\right\vert
)+V(x)\Phi(\left\vert u\right\vert )\right\}  \label{e29}%
\end{equation}
a.e.\ in $\mathbb{R}^{N}$, and $g\in L^{1}(\mathbb{R}^{N})$. Moreover, using
(\ref{e26}) we get%
\begin{align*}
\int_{\mathbb{R}^{N}}\left(  \phi(\left\vert \nabla u_{n}\right\vert
)\left\vert \nabla u_{n}\right\vert ^{2}+V(x)\phi(\left\vert u_{n}\right\vert
)\left\vert u_{n}\right\vert ^{2}\right)   &  =\left\langle \mathcal{A}%
(u_{n}),u_{n}\right\rangle \rightarrow\left\langle \mathcal{A}%
(u),u\right\rangle \\
&  =\int_{\mathbb{R}^{N}}\left(  \phi(\left\vert \nabla u\right\vert
)\left\vert \nabla u\right\vert ^{2}+V(x)\phi(\left\vert u\right\vert
)\left\vert u\right\vert ^{2}\right)  \text{,}%
\end{align*}
which implies
\begin{equation}
\int_{\mathbb{R}^{N}}g_{n}\rightarrow\int_{\mathbb{R}^{N}}g\text{.}
\label{e27}%
\end{equation}
Now, (\ref{e28}), (\ref{e29}), (\ref{e27}) and the generalized Lebesgue
dominated theorem gives
\[
\int_{\mathbb{R}^{N}}\left[  \Phi(\left\vert \nabla u_{n}-\nabla u\right\vert
)+V(x)\Phi(\left\vert u_{n}-u\right\vert )\right]  =\int_{\mathbb{R}^{N}}%
f_{n}\rightarrow0\text{,}%
\]
that is $u_{n}\rightarrow u$ in $X$.
\end{proof}

\begin{lemma}
\label{l6}The functional $\mathcal{F}:X\rightarrow\mathbb{R}$ defined by
\[
\mathcal{F}( u) =\int_{\mathbb{R}^{N}}F( u)
\]
is weakly-strongly continuous, that is, if $u_{n}\rightharpoonup u$ in $X$,
then $\mathcal{F}( u_{n}) \rightarrow\mathcal{F}( u) $.
\end{lemma}

\begin{proof}
Suppose $\left\{  u_{n}\right\}  \subset X$ satisfies $u_{n}\rightharpoonup u$
in $X$. Then $u_{n}\rightarrow u$ in $L^{\Phi}(\mathbb{R}^{N})$, by Lemma
\ref{l0}. Thus $\Phi(\left\vert u_{n}-u\right\vert )\rightarrow0$ in
$L^{1}(\mathbb{R}^{N})$, which also implies
\[
\Phi(2\left\vert u_{n}-u\right\vert )\rightarrow0\qquad \text{in }
L^{1}(\mathbb{R}^{N})\text{.}
\]
By \cite[Theorem 4.9]{MR2759829}, there exists $k\in
L^{1}(\mathbb{R}^{N})$ such that up to a subsequence,%
\[
\Phi(2\left\vert u_{n}-u\right\vert )\rightarrow0\text{\quad a.e.\ in
}\mathbb{R}^{N}\text{,\qquad}\Phi(2\left\vert u_{n}-u\right\vert )\leq
k\text{\quad a.e.\ in }\mathbb{R}^{N}\text{.}%
\]
By the monotonicity and convexity of $\Phi$,%
\begin{align*}
\Phi(\left\vert u_{n}\right\vert )  &  \leq\Phi(\left\vert u_{n}-u\right\vert
+\left\vert u\right\vert )\\
&  \leq\frac{1}{2}\Phi(2\left\vert u_{n}-u\right\vert )+\frac{1}{2}%
\Phi(2\left\vert u\right\vert )\leq\frac{1}{2}(k+\Phi(2\left\vert u\right\vert
))\text{.}%
\end{align*}
Since $k+\Phi(2\left\vert u\right\vert )\in L^{1}(\mathbb{R}^{N})$, and
$\Phi(\left\vert u_{n}\right\vert )\rightarrow\Phi(\left\vert u\right\vert )$
a.e.\ in $\mathbb{R}^{N}$, we deduce%
\begin{equation}
\int_{\mathbb{R}^{N}}\Phi(\left\vert u_{n}\right\vert )\rightarrow
\int_{\mathbb{R}^{N}}\Phi(\left\vert u\right\vert )\text{.} \label{e30}%
\end{equation}

For any $\varepsilon>0$, choose $C_{\varepsilon}>0$ such that%
\[
\left\vert F(t)\right\vert \leq\varepsilon\Phi_{\ast}(\left\vert t\right\vert
)+C_{\varepsilon}\Phi(\left\vert t\right\vert )\text{.}%
\]
Then we have%
\begin{align*}
\left\vert \mathcal{F}(u_{n})-\mathcal{F}(u)\right\vert  &  =\left\vert
\int_{\mathbb{R}^{N}}F(u_{n})-\int_{\mathbb{R}^{N}}F(u)\right\vert \\
&  \leq\varepsilon\left(  \int_{\mathbb{R}^{N}}\Phi_{\ast}(\left\vert
u_{n}\right\vert )+\int_{\mathbb{R}^{N}}\Phi_{\ast}(\left\vert u\right\vert
)\right)  +C_{\varepsilon}\left(  \int_{\mathbb{R}^{N}}\Phi(\left\vert
u_{n}\right\vert )-\int_{\mathbb{R}^{N}}\Phi(\left\vert u\right\vert )\right)
\text{.}%
\end{align*}
Using (\ref{e30}) we get%
\begin{equation}
\varlimsup_{n\rightarrow\infty}\left\vert \mathcal{F}(u_{n})-\mathcal{F}%
(u)\right\vert \leq\varepsilon\left(  \int_{\mathbb{R}^{N}}\Phi_{\ast
}(\left\vert u_{n}\right\vert )+\int_{\mathbb{R}^{N}}\Phi_{\ast}(\left\vert
u\right\vert )\right)  \text{.}%
\label{exd}
\end{equation}
Since $\left\{  u_{n}\right\}  $ in bounded in $X$, by the continuous embedding
$X\hookrightarrow L^{\Phi_{\ast}}(\mathbb{R}^{N})$ we see that $\left\{  u_{n}\right\}  $ is bounded in $L^{\Phi_{\ast}}(\mathbb{R}^{N})$, letting
$\varepsilon\rightarrow0$ in (\ref{exd}) we deduce $\mathcal{F}(u_{n})\rightarrow
\mathcal{F}(u)$.
\end{proof}

Now we are ready to prove the second part of Theorem \ref{t1}. We need the
following symmetric mountain pass theorem due to Ambrosetti-Rabinowitz
\cite{MR0370183}.

\begin{proposition}
[{\cite[Theorem 9.12]{MR845785}}]\label{p1}Let $X=Y\oplus Z$ be an infinite
dimensional Banach space with $\dim Y<\infty$. Suppose $\mathcal{J}\in C^{1}(
X) $ is even, satisfies $( PS) $, $\mathcal{J}( 0) =0$ and

\begin{description}
\item[$(1)$] for some $\rho>0$, $\inf_{\partial B_{\rho}\cap Z}\mathcal{J}>0$,
where $B_{\rho}=\left\{  \left.  u\in X\right\vert \,\left\Vert u\right\Vert
<\rho\right\}  $,

\item[$( 2) $] for any finite dimensional subspace $W\subset X$, there is an
$R=R( W) $ such that $\mathcal{J}\leq0$ on $W\backslash B_{R( W) }$,
\end{description}
then $\mathcal{J}$ has a sequence of critical values $c_{j}\rightarrow+\infty$.
\end{proposition}

\subsection{Proof of Theorem \ref{t1} (2)}

We know that the $C^{1}$-functional $\mathcal{J}$ given in (\ref{ecc}) is
even, satisfies $(PS)$ and $\mathcal{J}(0)=0$. To get an unbounded sequence of
critical values of $\mathcal{J}$, it suffices to verify the two assumptions in
Proposition \ref{p1}.

\emph{Verification of }$(1)$. Since $X$ is separable reflexive Banach space,
there exist $\left\{  e_{i}\right\}  _{1}^{\infty}\subset X$ and
$\{f^{i}\}_{1}^{\infty}\subset X^{\ast}$ such that $\big\langle f^{i}%
,e_{j}\big\rangle=\delta_{j}^{i}$ and%
\[
X=\overline{\operatorname*{span}}\left\{  \left.  e_{i}\right\vert
\,i\geq1\right\}  \text{,\qquad}X^{\ast}=\overline{\operatorname*{span}%
}^{w^{\ast}}\left\{  \left.  f^{i}\right\vert \,i\geq1\right\}  \text{.}%
\]
Let%
\[
Y_{k}=\overline{\operatorname*{span}}\left\{  \left.  e_{i}\right\vert
\,i<k\right\}  \text{,\qquad}Z_{k}=\overline{\operatorname*{span}}\left\{
\left.  e_{i}\right\vert \,i\geq k\right\}  \text{.}%
\]
In Lemma \ref{l6} we have proved that the functional $\mathcal{F}$ is
weakly-strongly continuous. Therefore, by \cite[Lemma 3.3]{MR2092084} we see
that%
\begin{equation}
\beta_{k}=\sup_{u\in\partial B_{1}\cap Z_{k}}\left\vert \mathcal{F}%
(u)\right\vert \rightarrow0\text{.} \label{e31}%
\end{equation}

For $u\in\partial B_{1}$, we have $\left\vert \nabla u\right\vert _{\Phi}%
\leq1$ and $\left\vert u\right\vert _{\Phi,V}\leq1$. Hence there exists a
constant $c>0$ such that%
\[
\int_{\mathbb{R}^{N}}\Phi(\left\vert \nabla u\right\vert )+\int_{\mathbb{R}%
^{N}}V(x)\Phi(\left\vert u\right\vert )\geq\left\vert \nabla u\right\vert
_{\Phi}^{\ell}+\left\vert u\right\vert _{\Phi,V}^{\ell}\geq c\text{.}%
\]
Using (\ref{e31}), we can choose $k$ such that $\beta_{k}<c$. Set $Z=Z_{k}$
and $Y=Y_{k}$. Then $\dim Y<\infty$, $X=Y\oplus Z$, for $u\in\partial
B_{1}\cap Z$ we have%
\[
\mathcal{J}(u)=\int_{\mathbb{R}^{N}}\Phi(\left\vert \nabla u\right\vert
)+\int_{\mathbb{R}^{N}}V(x)\Phi(\left\vert u\right\vert )-\mathcal{F}(u)\geq
c-\beta_{k}>0\text{.}%
\]
This verifies condition $(1)$ of Proposition \ref{p1}.

\emph{Verification of }$(2)$. Because $\theta>m$, condition $(f_{2})$ implies%
\[
\lim_{\left\vert t\right\vert \rightarrow\infty}\frac{F(t)}{\left\vert
t\right\vert ^{m}}=+\infty\text{.}%
\]
Let $W$ be any given finite dimensional subspace of $X$ and $\left\{  u_{n}\right\}  $
be a sequence in $W$ such that $\left\Vert u_{n}\right\Vert \rightarrow\infty
$. Then%
\[
v_{n}=\frac{u_{n}}{\left\Vert u_{n}\right\Vert }\rightarrow v
\]
for some $v\in W\cap\partial B_{1}$. For $x\in\left\{  v\neq0\right\}  $ we
have%
\[
\left\vert u_{n}(x)\right\vert =\left\Vert u_{n}\right\Vert \left\vert
v_{n}(x)\right\vert \rightarrow+\infty\text{.}%
\]
Applying the Fatou lemma and noting $F\ge0$, we deduce%
\begin{align*}
\frac{1}{\left\Vert u_{n}\right\Vert ^{m}}\int_{\mathbb{R}^{N}}F(u_{n})  &
\geq\frac{1}{\left\Vert u_{n}\right\Vert ^{m}}\int_{v\neq0}F(u_{n})\\
&  =\int_{v\neq0}\frac{F(u_{n})}{\left\vert u_{n}\right\vert }\left\vert
v_{n}\right\vert \rightarrow+\infty\text{.}%
\end{align*}
Consequently, applying Lemma \ref{l5} we get
\begin{align*}
\mathcal{J}(u_{n})  &  =\int_{\mathbb{R}^{N}}\Phi(\left\vert \nabla
u_{n}\right\vert )+\int_{\mathbb{R}^{N}}V(x)\Phi(\left\vert u_{n}\right\vert
)-\int_{\mathbb{R}^{N}}F(u_{n})\\
&  \leq\xi_{1}(\left\vert \nabla u_{n}\right\vert _{\Phi})+\xi_{1}(\left\vert
u\right\vert _{\Phi,V})-\int_{\mathbb{R}^{N}}F(u_{n})\\
&  \leq\left\vert \nabla u_{n}\right\vert _{\Phi}^{m}+\left\vert \nabla
u_{n}\right\vert _{\Phi}^{\ell}+\left\vert u_{n}\right\vert _{\Phi,V}%
^{m}+\left\vert u_{n}\right\vert _{\Phi,V}^{\ell}-\int_{\mathbb{R}^{N}}%
F(u_{n})\\
&  \leq2\left(  \left\Vert u_{n}\right\Vert ^{m}+\left\Vert u_{n}\right\Vert
^{\ell}\right)  -\int_{\mathbb{R}^{N}}F(u_{n})\\
&  =2\left\Vert u_{n}\right\Vert ^{m}\left(  1+\left\Vert u_{n}\right\Vert
^{\ell-m}-\frac{1}{\left\Vert u_{n}\right\Vert ^{m}}\int_{\mathbb{R}^{N}%
}F(u_{n})\right)  \rightarrow-\infty
\end{align*}
because $\ell\leq m$. Hence condition $(2)$ of Proposition \ref{p1} is
verified, and the proof of Theorem \ref{t1} (2) is completed.

\end{document}